\documentclass[12pt]{amsart}
 \usepackage{mathtools}
\mathtoolsset{showonlyrefs,showmanualtags}

\usepackage[backrefs]{amsrefs}

 \usepackage[top=1.5in, bottom=1.5in, left=1.25in, right=1.25in]	{geometry}
\usepackage[all]{xy}
\usepackage{amsmath}
\usepackage{amssymb}
\usepackage{amsthm}
\usepackage{amscd}

\numberwithin{equation}{section}

\newtheorem{theorem}[equation]{Theorem}

\newtheorem{proposition}[equation]{Proposition}
\newtheorem{cor}[equation]{Corollary}
\newtheorem{lemma}[equation]{Lemma}

\newtheorem{remark}[equation]{Remark}

\usepackage{hyperref} %,hypertexnames=false,colorlinks,[pagebackref]
\hypersetup{
    colorlinks=true,       % false: boxed links; true: colored links
    linkcolor=blue,          % color of internal links
    citecolor=magenta,        % color of links to bibliography
    filecolor=magenta,      % color of file links
    urlcolor=cyan           % color of external links
}

\begin{document}
\title[Non-Linear Roth]{A Non-Linear Roth Theorem for Sets of Positive Density}
%Enter your title between curly braces
\author{Ben Krause}
\address{
Department of Mathematics,
Caltech \\
Pasadena, CA 91125}
\email{benkrause2323@gmail.com}
\date{\today}

\begin{abstract}
Suppose that $A \subset \mathbb{R}$ has positive upper density, 
\[ \limsup_{|I| \to \infty} \frac{|A \cap I|}{|I|} = \delta > 0,\]
and $P(t) \in \mathbb{R}[t]$ is a polynomial with no constant or linear term, or more generally a \emph{non-flat} curve. Then for any $R_0 \leq R$ sufficiently large, there exists some $x_R \in A$ so that
\[ \inf_{R_0 \leq T \leq R} \frac{|\{ 0 \leq t < T : x_R - t \in A, \ x_R - P(t) \in A \}|}{T} \geq c_P \cdot \delta^2 \]
for some absolute constant $c_P > 0$, that depends only on $P$.
\end{abstract}

\maketitle

% \setcounter{tocdepth}{1}
%\tableofcontents 

\section{Introduction}
A beautiful result in Euclidean Ramsey theory, due to Furstenberg, Katznelson, and Weiss, \cite{FKW}, concerns the presence of additive structure in sets of positive density inside of the plane, i.e.\ those (measurable) sets for which
\[ d^*(A) := \limsup_{|Q| \to \infty} \frac{ |A \cap Q| }{|Q|} > 0;\]
here $\{Q\}$ are axis-parallel cubes.
\begin{theorem}\label{t:FKW}
Suppose that $A \subset \mathbb{R}^2$ has $d^*(A) > 0$. Then there exists a threshold $r_0$ so that for every $r \geq r_0$ there exists $x_r,y_r \in A$ so that
\[ |x_r - y_r| = r.\]
\end{theorem}
This result was first proven by ergodic theoretic techniques, but it has since been recovered by different methods: geometric \cite{FM} and probabilistic \cite{Q}, and -- most significant for this paper -- Fourier analytic \cite{B}.

Indeed, in his paper, \emph{A Szemer\'{e}di type theorem for sets of positive density in $\mathbb{R}^k$} \cite{B}, Bourgain developed a powerful but elementary Fourier analytic approach 
through which he not only recovered Theorem \ref{t:FKW}, but also was able to reveal the presence of additional additive structure inside of dense subsets of the plane, via the following pinned variant of Theorem \ref{t:FKW}.
\begin{theorem}\label{t:B}
Suppose that $A \subset \mathbb{R}^2$ has $d^*(A) > 0$. Then there exists a threshold $r_0$ so that for every $r \geq r_0$ there exists $x_r \in A$ so that for all $r_0 \leq s \leq r$, there exists some $y_s \in A$ so that
\[ |x_r - y_s| = s.\]
\end{theorem}

To the extent that the core of Bourgain's argument amounted -- essentially -- to appropriate applications of the uncertainty principle, his methods have proven quite robust in addressing other problems in Euclidean Ramsey theory, see for instance \cite{CMP}, \cite{D+}, \cite{HLM}, \cite{LM3}, as well as in the discrete context, \cite{LM}, \cite{LM1}, \cite{LM2}, \cite{M}.

In this note, we apply his method in the {bi-linear} setting, as we consider the issue of Non-Linear Roth's Theorem (in the Euclidean context). 

The first progress towards understanding non-linear Roth-type patterns
\[ \{ x, x + t, x+P(t) \}, \ P \in \mathbb{R}[-] \]
inside non-trivial subsets of the real line was made by Bourgain, \cite{Broth}. In particular, he proved the following theorem.
\begin{theorem}\label{t:Broth}
Let $N \geq 1$, $\delta > 0$ be small, and let $A \subset [0,N]$ have $|A| \geq \delta N$, and $d \geq 2$ be arbitrary. Then there exists 
\[ x,x-t,x-t^d \in A\]
for some $t \geq c_\delta \cdot N^{1/d}$, for some absolute constant $c_\delta$.
\end{theorem}
In \cite{DGR}, this result was extended to handle polynomial curves $P(t) \in \mathbb{R}[t]$ without constant or linear terms.

Both results are a consequence of the following proposition. With
\[ B_r(f,g)(x) := B_r^P(f,g)(x) := \frac{1}{r} \int_0^r f(x-t)g(x-P(t)) \ dt,\]
\begin{proposition}
Let $\delta > 0$ be small, and let $A \subset [0,1]$ have $|A| \geq \delta.$ If $P(t)$ is a polynomial as above, then there exists a lower bound
\[ \langle \mathbf{1}_A , B_1(\mathbf{1}_A, \mathbf{1}_A) \rangle \geq c_{\delta,\|P\|} ,\]
for some $c_{\delta,\|P\|} > 0$ which depends on $\delta$ and $\|P\|$, the $\ell^1$ sum of the coefficients of $P$.
\end{proposition}

In this paper, we amplify this proposition as follows. First, we allow for the presence of \emph{non-flat} $P$, functions introduced and studied by Lie in his treatment of bilinear Hilbert transforms with curvature, \cite{Lie1} and \cite{Lie2}, in the definition of $B_r = B_r^P$. Roughly speaking, non-flat curves are locally differentiable curves which do not ``resemble a line" near the origin or $\infty$. For instance, in addition to the types of polynomials discussed above, real analytic functions which vanish to degree two at the origin are non-flat, as are real laurent polynomials of the form 
\[ \sum_{-n}^m a_j t^j, \ a_{-m}, a_n \neq 0, \ n,m \geq 2,\]
or even
linear combinations of functions of the form 
\[ |t|^\alpha | \log |t| |^\beta, \ \alpha,\beta \in \mathbb{R}, \ \alpha \neq 0,1.\]
For a more precise definition, see \cite[\S 2]{Lie1}.

Our result is then the following.

\begin{proposition}\label{p:Rothish}
Suppose that $(0,1] = \bigcup_{J \in \mathcal{P}} J$ is an \emph{admissible} partition of intervals, in that each $J$ contains at least one dyadic rational of the form $2^{-k}$. If $A \subset [0,1]$ has $|A| \geq \delta$, then there exists a subset $\mathcal{Q} \subset \mathcal{P}$ so that 
\begin{equation}\label{e:cp}
\inf_{J \in \mathcal{Q} } |\langle \mathbf{1}_A, \inf_{r \in J} B_r(\mathbf{1}_A,\mathbf{1}_A) \rangle| \geq c_{P} \cdot \delta^3 
\end{equation}
for some absolute $c_{P} > 0.$ Moreover, there exists some absolute constant $C_{P}$ so that
\[
|\{ J : J \in \mathcal{P} \smallsetminus \mathcal{Q}\}| \leq C_{P} \cdot \delta^{-5} \log (\delta^{-1}) .
\]
\end{proposition}
\begin{remark}
For polynomial $P$, both $c_P,C_P$ depend only on $\|P\|$.
\end{remark}
The following corollary, a quantitative improvement over \cite{DGR}, immediately presents.

\begin{cor}\label{c:Rothish}
In the setting of Proposition \ref{p:Rothish}, for any $\epsilon >0$, there exists an absolute constant $c_{\epsilon, P}$ so that 
\[ \inf_{1 \geq r > 0} |\langle \mathbf{1}_A, B_r(\mathbf{1}_A,\mathbf{1}_A) \rangle| \geq \frac{ c_{\epsilon,P} } { 2^{\delta^{-5 - \epsilon}} } \cdot c_{P} \cdot \delta^3, \]
where $c_{P}$ is as in \eqref{e:cp}, and similarly $c_{\epsilon,P}$ is determined by $\|P\|$ for polynomial $P$.
\end{cor}

Our main result follows from Proposition \ref{p:Rothish} by arguing by contradiction and rescaling as in \cite{B}; this argument is by now standard, and has appeared in the above-discussed works, and so we omit it.

\begin{theorem}\label{t:main}
Suppose $A \subset \mathbb{R}$ has positive upper density $d^*(A) = \delta$. Then for every $R \geq R_0$ sufficiently large there exists an $ x_R \in A$ so that
\[ \inf_{R_0 \leq T \leq R} \frac{|\{ 0 \leq t \leq T : x_R -t \in A, \ x_R - P(t) \in A \}|}{T} \gtrsim_{P} \delta^2.\]
\end{theorem}

For ease of presentation, we will establish our main results only in the case of $P(t) = t^2$, as the complications that arise in increasing the generality are essentially notational.

\subsection{Acknowledgement}
This paper, like many of my papers, was inspired by the work of Jean Bourgain; his impact on my mathematics has been profound.

\subsection{Notation}
Here and throughout, $e(t) := e^{2\pi i t}$. Throughout, $C$ will be a large number which may change from line to line. 

We will henceforth \emph{re-define} 
\begin{equation}\label{e:Br}
B_k (f,g)(x):= \int f(x-t) g(x-t^2) \rho_k(t) \ dt
\end{equation} 
where $\rho_k(t):= 2^k \rho(2^k t)$, and $\rho$ is an appropriate bump function. 
%For notational ease, we will conflate the Fourier transform of the convolution operator with its symbol, thus
%\[ \widehat{B_k}(\xi,\zeta) := \int e(-\xi \cdot 2^{-k} t - \zeta \cdot 2^{-2k} t^2 ) \rho(t) \ dt.\]

We will also make use of certain Fourier projection operators: we let $\phi$ denote a Schwartz function which satisfies
\[ \mathbf{1}_{|\xi| \leq 1/8} \leq \widehat{\phi} \leq \mathbf{1}_{|\xi| \leq 1/2},
\]
and set $\phi_k(t) := 2^k \phi(2^k t)$.

We will make use of the modified Vinogradov notation. We use $X \lesssim Y$, or $Y \gtrsim X$, to denote the estimate $X \leq CY$ for an absolute constant $C$. We use $X \approx Y$ as shorthand for $Y \lesssim X \lesssim Y$. We also make use of big-O notation: we let $O(Y )$ denote a quantity that is $\lesssim Y$. If we need $C$ to depend on a parameter, we shall indicate this by subscripts, thus for instance $X \lesssim_{\delta} Y$ denotes the estimate $X \leq C_{\delta} Y$
for some $C_{\delta}$ depending on $\delta$. We analogously define $O_{\delta}(Y)$.

\section{Preliminaries}
Before turning to the proof, we need to collect various results from Euclidean harmonic analysis. The first is essentially a martingale inequality, and was observed by Bourgain in \cite{Broth}; see \cite{DGR} for a proof.
m
\begin{lemma}\label{l:mart}
Suppose that $1 \geq f \geq 0$ is supported in $[0,1]$. Then for any $r,s > 0$,
%\[ \int f \cdot \rho_r*f 
%\gtrsim \left( \int f \right)^2 \]
%and
\[ \int f \cdot \rho_r*f \cdot \rho_s*f
\gtrsim \left( \int f \right)^3.\]
\end{lemma}

The second is a special case of a result of Li and Xiao, \cite{LX}; it's extension to non-flat curves is announced there, but a full proof can be found in forthcoming work \cite{GKL}. 

\begin{theorem}\label{t:Bil2}
Suppose that $P$ is as above. Then
\[ \| \sup_{k} B_k(f,g) \|_1 \lesssim_{P} \|f\|_2 \|g\|_2.\]
\end{theorem}

The key ingredient in proving this theorem is obtaining so-called ``scale-type" decay. Some notation: with $\Psi$ a smooth approximation of the indicator function of an annulus, define via the Fourier transform
\[ \widehat{f_m}(\xi) := \hat{f}(\xi) \cdot \Psi(2^{-m} \xi).\] 
The following is the key proposition; it first appeared essentially as in \cite[Theorem 2]{Lie1}, see also \cite[Propositions 3,4]{LX}.

\begin{proposition}\label{p:scale}
The following estimate holds for any $|p| \lesssim 1$
\[ \| B_k(f_{k+m},g_{2k+m+p}) \|_1 \lesssim_P 2^{-\epsilon m} \|f\|_2 \|g\|_2.\]
\end{proposition}

With these preliminaries in mind, we turn to the proof.

\section{Proof of Proposition \ref{p:Rothish} }

Proposition \ref{p:Rothish} will follow directly from the following proposition.

\begin{proposition}\label{p:key}
Suppose $A \subset [0,1]$ has $|A| = \delta$, and $k \geq l \gg \log(\delta^{-1})$ is sufficiently large.
Then there exists an absolute constant $1 \gg c_0 > 0$ so that if the following upper bound holds,
\begin{equation}\label{e:lb} \int_A \inf_{l \leq r \leq k} B_r(\mathbf{1}_A,\mathbf{1}_A)(x) \ll c_0 \delta^3, 
\end{equation}
then
\begin{equation}\label{e:lowerbound}
\| \widehat{ \mathbf{1}_A }(\xi) \cdot  \mathbf{1}_{\delta^C 2^l \lesssim |\xi| \lesssim \delta^{-C} 2^k} \|_2 + \| \widehat{ \mathbf{1}_A }(\xi) \cdot  \mathbf{1}_{\delta^C 2^{2l} \lesssim |\xi| \lesssim \delta^{-C} 2^{2k}} \|_2 \gtrsim \delta^3.
\end{equation}
\end{proposition}

Before turning to the proof, we will make use of the following splitting. Set $k_\delta := k + C \log \delta^{-1}$ and $l_\delta := l - C \log \delta^{-1}$, and assume that we are interested in decomposing
\[ B_r(f,g) \]
for some $l \leq r \leq k$.
We split 
\[ f = f_L + f_M + f_H,\]
where
\begin{equation}\label{e:decomp}
\widehat{f_L} := \hat{f} \cdot \widehat{\phi_{l_\delta}} \text{ and } 
\widehat{f_H} := \hat{f} \cdot ( 1- \widehat{\phi_{k_\delta}}),
\end{equation}
and
\[ g = g_L + g_M + g_H,\]
where
\begin{equation}\label{e:decomp'}
\widehat{g_L} := \hat{g} \cdot \widehat{\phi_{2 l_\delta}} \text{ and } 
\widehat{g_H} := \hat{g} \cdot ( 1- \widehat{\phi_{2 k_\delta}}),
\end{equation}

Furthermore, for ease of presentation, we will drop all terms which contribute $\lesssim \delta^C$ as negligible.

With this in mind, we turn to the proof.
\begin{proof}[Proof of Proposition \ref{p:key}]
Suppose that \eqref{e:lb} holds. Then, with $B := [0,1] \smallsetminus A$ we have the following lower bound, provided that $l \gg \log(\delta^{-1})$ is sufficiently large:
\begin{equation}\label{e:3term} \int_A \sup_{l \leq r \leq k} |B_r(\mathbf{1}_{B},\mathbf{1}_{B})| + \sup_{l \leq r \leq k} |B_r(\mathbf{1}_{A},\mathbf{1}_{B})| + \sup_{l \leq r \leq k} |B_r(\mathbf{1}_{B},\mathbf{1}_{A})| \geq (1 - \frac{c_0}{1000} \cdot \delta^2 ) \cdot \delta.
\end{equation}
For each $f,g$, we decompose 
\[ \aligned 
B_r(f,g) & =  B_r(f_L,g_L)\\
& \qquad + 
B_r(f_L,g_M + g_H) + 
B_r(f_M,g) \\
& \qquad \qquad + B_r(f_H,g_L + g_M) + B_r(f_H,g_H). \endaligned \]
For $l \leq r \leq k$, $B_r(f_L,g_L) = f_L \cdot g_L$, up to pointwise errors of $\delta^C$; by Lemma \ref{l:mart} and trivial geometric considerations, taking into account the large size of $l\gg_\delta 1$, we deduce that
\[ \sum_{C,C' \in \{A,B\}^2 \smallsetminus \{A,A\}} \int_A  
\sup_{l \leq r \leq k} |B_r((\mathbf{1}_{C})_L,(\mathbf{1}_{C'})_L)| \leq \delta - c \delta^3,\]
where $c$ is essentially given by Lemma \ref{l:mart}.
We also observe that
\[ B_r(f_L,g_M + g_H) = f_L \cdot \mathcal{M} g_M, \]
again up to pointwise errors on the order of $\delta^C$ on $[0,1]$; here, $\mathcal{M}$ denotes an appropriate maximal function, pointwise bounded by the Hardy-Littlewood maximal function. Again using the large size of $l$, we have that $(\mathbf{1}_B)_M = -(\mathbf{1}_A)_M$ up to errors which have $L^1$ norm $\lesssim \delta^C$ when restricted to $[0,1]$, by Cauchy-Schwartz. Next, 
\begin{equation}\label{e:PP} B_r(f_H,g_L+g_M) = 2^{-\epsilon(k-r)} \cdot \delta^C \cdot \Pi_{r,k}(f_H,g_L + g_M), 
\end{equation}
for some paraproduct $\Pi_{r,k}$, where $\|\Pi_{r,k} \|_{L^2 \times L^2 \to L^1} \leq C_{k-r} \leq C$ for all $r,k$ and some absolute $C$, see \cite{Lie1}. Thus
\[ \int_A \sup_{l \leq r \leq k} |B_r(f_H,g_L +g_M)| \ll \delta^C. \]
It remains to analyze $B_r(f_H,g_H)$. If we express
\[ B_r(f_H,g_H) = \sum_{m,n \gg \log(\delta^{-1}), |m-n| \lesssim 1 } B_r( f_{k+m}, g_{2k+n}) + 
\sum_{m,n \gg \log(\delta^{-1}), |m-n| \gg 1 } B_r( f_{k+m}, g_{2k+n}), \]
then the second term is a sum of rapidly decaying paraproducts the sum of whose $L^1$ norms are $O(2^{-\epsilon(k-r)} \cdot \delta^C)$, while the first term has $L^1$ norm bounded by $2^{-\epsilon(k-r)} \cdot \delta^C$ by 
Proposition \ref{p:scale}. In particular, the lower bound \eqref{e:3term} forces
\[ \sum_{C,C' \in \{A,B\}^2 \smallsetminus \{A,A\}} 
\int_A (\mathbf{1}_C)_L \cdot \mathcal{M}\left( (\mathbf{1}_{C'})_M \right)+ \int_A \sup_r \left| B_r((\mathbf{1}_C)_M,\mathbf{1}_{C'}) \right| \gtrsim \delta^3.\]
By pointwise considerations, we may refine this to
\[ 
\int_A \mathcal{M} \left( (\mathbf{1}_{A})_M\right) + \int_A \sup_r \left|  B_r((\mathbf{1}_A)_M,\mathbf{1}_{A}) \right| +
\int_A \sup_r \left|  B_r((\mathbf{1}_A)_M,\mathbf{1}_{B}) \right|
\gtrsim \delta^3,\]
which yields the result by Cauchy-Schwartz and Theorem \ref{t:Bil2}.
\end{proof}

\typeout{get arXiv to do 4 passes: Label(s) may have changed. Rerun}

\end{document}